\documentclass[12pt]{article}

\usepackage{amsmath}
\usepackage{amsthm}
\usepackage{amssymb}
\usepackage{latexsym}
\usepackage{pstricks}
\usepackage{graphicx, pst-plot, pst-node, pst-text, pst-tree, young, ytableau}
\usepackage{titlefoot}
\usepackage{enumerate}
\usepackage{titlesec}
\usepackage{cancel}
\usepackage{rotating}
\usepackage[small,it]{caption}
\usepackage{color}
\definecolor{lightgray}{rgb}{0.8, 0.8, 0.8}
\definecolor{darkgray}{rgb}{0.7, 0.7, 0.7}
\definecolor{darkblue}{rgb}{0, 0, .4}

\usepackage[bookmarks]{hyperref}
\hypersetup{
        colorlinks=true,
        linkcolor=darkblue,
        anchorcolor=darkblue,
        citecolor=darkblue,
        urlcolor=darkblue,
        pdfpagemode=UseThumbs,
        pdftitle={An involution on involutions},
        pdfsubject={Combinatorics},
        pdfauthor={Miklo\'s Bo\'na and Rebecca Smith },
}

\newtheorem{theorem}{Theorem}[section]
\newtheorem{proposition}[theorem]{Proposition}

\newtheorem{lemma}[theorem]{Lemma}
\newtheorem{definition}[theorem]{Definition}
\newtheorem{corollary}[theorem]{Corollary}

\newtheorem{remark}[theorem]{Remark}

\newtheoremstyle{example}{\topsep}{\topsep}%
     {}
     {}
     {\bfseries}
     {.}
     {.5em}
     {\thmname{#1}\thmnumber{ #2}}
\theoremstyle{example}
\newtheorem{example}[theorem]{Example}

\newtheoremstyle{negexample}{\topsep}{\topsep}%
     {}
     {}
     {\bfseries}
     {.}
     {.5em}
     {\thmname{#1}\thmnumber{ #2}}
\theoremstyle{negexample}

\newcounter{todocounter}

\long\def\symbolfootnote[#1]#2{\begingroup%
\def\thefootnote{\fnsymbol{footnote}}\footnote[#1]{#2}\endgroup}

\makeatletter
\newcommand{\Rm}[1]{\expandafter\@slowromancap\romannumeral #1@}
\newfont{\footsc}{cmcsc10 at 8truept}
\newfont{\footbf}{cmbx10 at 8truept}
\newfont{\footrm}{cmr10 at 10truept}
\renewenvironment{abstract}%
                {
                  \begin{list}{}%
                     {\setlength{\rightmargin}{1in}%
                      \setlength{\leftmargin}{1in}}%
                   \item[]\ignorespaces\begin{small}}%
                 {\end{small}\unskip\end{list}}

\amssubj{05A17, 06A07}

\newpagestyle{main}[\small]{
        \headrule
        \sethead[\usepage][][]
        {\sc An Involution on Involutions and a Generalization of Layered Permutations}{}{\usepage}}

\title{\sc{An involution on involutions and a Generalization of Layered Permutations}}
\author{Mikl\'os B\'ona
\\[-0.25ex]
\small Department of Mathematics\\[-0.5ex]
\small University of Florida\\[-0.5ex]
\small Gainesville, Florida\\[15pt]
Rebecca Smith
\\[-0.25ex]
\small Department of Mathematics\\[-0.5ex]
\small SUNY Brockport\\[-0.5ex]
\small Brockport, New York\\[-1.5ex]
}

\titleformat{\section}
        {\large\sc}
        {\thesection.}{1em}{}   

\date{}

\begin{document}
\maketitle

\newcommand{\s}{\mathbf{s}}
\newcommand{\m}{\mathbf{m}}
\renewcommand{\t}{\mathbf{t}}
\renewcommand{\b}{\mathbf{b}}
\newcommand{\f}{\mathbf{f}}
\newcommand{\rev}{\operatorname{rev}}
\newcommand{\dual}{\operatorname{dual}}
\newcommand{\D}{\mathcal{D}}
\newcommand{\Av}{\operatorname{Av}}

\newcommand{\inp}{\textsf{i}}
\newcommand{\tra}{\textsf{t}}
\newcommand{\out}{\textsf{o}}

\newcommand{\R}{\stackrel{R}{\sim}}
\renewcommand{\L}{\stackrel{L}{\sim}}
\newcommand{\notR}{\stackrel{R}{\not\sim}}
\newcommand{\notL}{\stackrel{L}{\not\sim}}

\newcommand{\OEISlink}[1]{\href{http://oeis.org/#1}{#1}}
\newcommand{\OEISref}{\href{http://oeis.org/}{OEIS}~\cite{sloane:the-on-line-enc:}}
\newcommand{\OEIS}[1]{(Sequence \OEISlink{#1} in the \OEISref.)}

%
%
%
%
%
%
%
%

\def\sdwys #1{\xHyphenate#1$\wholeString}
\def\xHyphenate#1#2\wholeString {\if#1$%
\else\say{\ensuremath{#1}}\hspace{2pt}%
\takeTheRest#2\ofTheString
\fi}
\def\takeTheRest#1\ofTheString\fi
{\fi \xHyphenate#1\wholeString}
\def\say#1{\begin{turn}{-90}\ensuremath{#1}\end{turn}}

\begin{abstract} Taking transposes of Standard Young Tableaux defines a natural involution on the set $I(n)$ of involutions
of length $n$ via the
the Robinson-Schensted correspondence. In some cases, this involution can be defined without resorting to the
Robinson-Schensted correspondence. As a byproduct, we get an interesting generalization of layered permutations.
\end{abstract}

\section{Introduction}
The Robinson-Schensted (RS)~\cite{robinson:on-representations:,schensted:longest-increas:} correspondence bijectively maps a permutation $p$ into an ordered  pair of Standard Young Tableaux (SYT)
$(P(p), Q(p))$ on $n$ boxes and of the same shape. This bijection, and its consequences, have been analyzed from numerous perspectives in the last 50 years, see \cite{sagan} for a comprehensive treatment of these results. 

In particular, a famous result of Marcel-Paul Sch\"utzenberger \cite{schutz} shows that  if $p^{-1}$ denotes the inverse of $p$, then 
the RS correspondence maps $p^{-1}$ to the pair $(Q(p),P(p))$. Therefore, if $p$ is an involution, that is, 
if $p=p^{-1}$, then $(P(p),Q(p))=(Q(p),P(p))$, which implies that $P(p)=Q(p)$. In other words, the RS correspondence
defines a bijection from the set $I(n)$ of involutions of length $n$ to the set $SYT(n)$ of Standard Young Tableaux on
$n$ boxes.  This bijection has been studied for its own sake in \cite{beissinger}.  Therefore, with a slight abuse of language,
for involutions $p$, we will sometimes simply talk about the {\em tableau of $p$} when we mean the tableau $P(p)=Q(p)$.

Taking transposes defines a natural involution on $SYT(n)$. Therefore, as $I(n)$ is in bijection with $SYT(n)$, taking
transposes of the corresponding Standard Young Tableaux also defines an involution $f$ on $I(n)$. 
Interestingly, it seems that $f$ has not been the subject of many research papers.

Note that it is  well known  \cite{sagan} that if $p$ is any permutation and $p^{rev}$ is its reverse, 
then $P(p^{rev}) = P(p)^T$, where $P(p)^T$ is the transpose of the tableau $P(p)$. This implies that in the special
case when $p$ is an involution {\em and} $p^{rev}$ is also an involution, then $f(p)=p^{rev}$. However, 
if $p$ is an involution and $p^{rev}$ is not an involution, then there is no clear way of describing $f(p)$ without
using the machinery of the RS bijection. 

In this paper, our goal
is to describe the effect of $f$ on some subsets of $I(n)$ in terms of the involutions themselves only, that is,  {\em without resorting to the RS correspondence.}  In Section \ref{seclayered}, we give a description of $f$ for layered permutations, which are necessarily involutions, while in Section \ref{catalan}, we consider involutions that do not contain an increasing subsequence of length three, or a
decreasing subsequence of length three. Our results will be one-sided in that we do have a simple description of 
 $f(p)$ when $p$ is in a certain set $S$ of involutions, but we do not have a similarly simple description for $f(w)$ if 
$w\in f(S)$. This is, perhaps, not surprising, since the roles of rows and columns of SYT in  the RS correspondence are fundamentally different.  In Section \ref{generalized}, we use a new characterization of layered 
permutations that we prove in Section \ref{seclayered} to generalize these permutations in a natural way. 

\section{Layered permutations} \label{seclayered}
\begin{definition} A permutation $p=p_1p_2\cdots p_n$ is called {\em layered} if $p$ is a concatenation of decreasing
subsequences (the layers) so that for all $i$, each entry of the $i$th decreasing subsequence from the left
is smaller than each entry of the $(i+1)$st decreasing subsequence from the left. 
\end{definition}

\begin{example} Permutations 32154, 2154376, and 1234 are all layered. \end{example}
\begin{remark} \label{linvol} All layered permutations are {\em involutions}. \end{remark} 

There are many ways in which layered permutations can be characterized. For instance, they are precisely the 
permutations that avoid both 231 and 132 as patterns, and they are precisely the involutions whose patterns are
all involutions. In fact, all the patterns of layered patterns are themselves layered.  See Chapter 4 of \cite{combperm} for the relevant definitions in pattern avoidance. Layered permutations
have also been studied \cite{dan,dan1} from the perspective of their packing densities.  Another characterization of layered permutations will become important for us shortly.

As layered permutations of length $n$ are clearly in bijection with compositions of $n$, their number is $2^{n-1}$. So
if $L_n$ denotes the set of these permutations, $f(L_n)$ is a $2^{n-1}$-element subset of the symmetric group $S_n$. How can we describe this subset?

\begin{lemma} \label{layered}
Let $T$ be a SYT on $n$ boxes. Then $T$ is the tableau of a layered permutation $p$ if and only if $T$ satisfies the
following requirements. 

For all $i\in \{1,2,3,\ldots,n-1\}$, the entry $i+1$ is either
\begin{itemize} \item in the row directly below the row containing $i$, or
\item in the top row. 
\end{itemize}
\end{lemma}

\begin{proof}
First we prove the "only if" part,
  by induction on the number of layers in $p$, the case of one layer being obvious.  Let $p$ be a layered permutation with 
$\ell$ layers. As the entries of the last layer are larger than all preceding entries of $p$, their insertion does not displace
any of these preceding entries of $p$ from their place in the $P$-tableau of $p$, so they still satisfy the conditions of the 
lemma.  
Consider the last layer of $p$. If that layer consist of the entry $n$ only, then that entry $n$ will be inserted at the
end of the first row of the $P$-tableau of $p$, and we are done. If that last layer consists of the decreasing
subsequence $n(n-1)\cdots (n-a)$, then each of these $a+1$ entries will be inserted at the end of the first row of  the 
$P$-tableau of $p$, and the subsequently bumped one row lower by each entry that follows.  So indeed, the entries of the last layer will be positioned as
described in the first condition of the Lemma, and our induction proof is complete. 

In order to prove the "if" part, we can again argue by induction, or we can note that the number of SYT satisfying the
conditions of the Lemma is $2^{n-1}$, since any SYT on $n-1$ boxes that has the required properties can be completed
to a SYT on $n$ boxes having those properties in two different ways, either by placing $n$ at the end of the first row, or
by placing $n$ at the end of the row right below the row that contains $n-1$. Note that since no smaller entry can be directly below $n-1$, there is room in this lower row to place $n$.  So, the injective map from layered permutations with the given property must be surjective, hence it is bijective. 
\end{proof}

If a SYT satisfies the conditions of Lemma \ref{layered}, in other words, when it is the tableau of a layered permutation, 
we will call that tableau {\em layered}. 

Note in particular that the entries of any one layer of $p$ are in all distinct (and consecutive) rows of $P(p)=Q(p)$, 
the smallest entry of the layer in the first row, the second smallest entry of the layer in the second row, and so on.

\begin{example} \label{exalayered}
The tableau of the layered permutation $p =215439876$ is shown below 
\[
\ytableausetup{centertableaux}
\begin{ytableau}
1 & 3& 6 \\
2 & 4 & 7\\
5 & 8\\
9
\end{ytableau}.\]
\end{example}

Taking the transpose of a layered tableau $U$ will have an obvious effect on the properties proved in 
Lemma \ref{layered}. This motivates the following definition. 

\begin{definition} If $U$ is a SYT in which for all $i\leq n-1$, the entry $i+1$ is
\begin{itemize}
\item in the first column, or 
\item in the column immediately on the right of the column that contains $i$,
\end{itemize} then we say that $U$ satisfies the {\em transposed layer} conditions.

If the tableau of the involution $p$ satisfies the transposed layer condition, then we will say that $p$ satisfies
the transposed layer condition. 
\end{definition}

For instance,  taking the transpose of the tableau of Example \ref{exalayered}, we get the tableau

\[
\ytableausetup{centertableaux}
\begin{ytableau}
1 & 2 & 5  & 9\\
3&  4 & 8\\
6 & 7\\
\end{ytableau}.\]

Applying the inverse of the RS correspondence to the SYT above, we get that $f(p)=673481259$.
We notice that all three layers of $p$ are reversed in $f(p)$, but the entries of each layer of $p$ are no longer in 
consecutive positions in $f(p)$. We will explain that this is not by accident, and discuss what part of the structure of a
layered permutation is preserved by our involution $f$.

In a permutation $p=p_1p_2\cdots p_n$, we say that $i$ is a {\em descent} if $p_i>p_{i+1}$. In a Standard Young Tableaux $T$, we say that
$i$ is a {\em descent} if $i+1$ occurs in a row that is strictly below the row containing $i$. For instance, in the last displayed tableau, the entries 2 and 5 are descents. 

The following well-known fact is easy to prove. (See \cite{combperm}, Theorem 7.15 for a proof.)

\begin{proposition} \label{descent}
For any permutation $p$, the position $i$ is a descent in $p$ if and only if the entry $i$ is a descent of $Q(p)$. 
\end{proposition}

\begin{remark}  \label{inverse} Note that the two equivalent statements of Proposition \ref{descent} are also equivalent to the statement
that the entry $i+1$ precedes the entry $i$ in $p^{-1}$. In particular, if  $p=p_1p_2\cdots p_n$ is an involution, then this just means that
$i+1$ precedes $i$ in $p$. \end{remark}

Let $p$ be any permutation. We define a {\em jog} in $p$ as a maximal increasing subsequence of consecutive integers. For instance, $p=3614725$ has jogs 12, 345, and 67. Each permutation decomposes into its jogs. 
Note that if $p=p_1p_2\cdots p_n$ is an involution, then the fact that the entries $j, j+1,\cdots, j+a$
form a jog is equivalent to the fact that $p_jp_{j+1}\cdots p_{j+a}$ is an increasing subsequence of consecutive 
entries (position-wise) in $p$ that cannot be extended on either side. This is often described by saying that $p_jp_{j+1}\cdots p_{j+a}$ is a {\em run} or an {\em ascending run} in $p$.

\begin{corollary} 
\label{easyfact} If $w$ is a layered permutation, and the entries $a, a+1, \cdots ,a+b$ form a layer in $w$, then 
the entries $a,a+1,\cdots ,a+b$ form a jog in $f(w)$. 
\end{corollary}

\begin{proof} If the  $a, a+1, \cdots ,a+b$ form a layer in $w$, then the positions $a,a+1,\cdots ,a+b-1$ are descents
in $w$, so those entries are descents in $Q(w)$, which implies that they are {\em not} descents in $Q(w)^T$, and hence
they are not descents in $f(w)$. As $f(w)$ is an involution, it follows that the
entries $a,a+1,\cdots ,a+b$ form a non-extendible increasing subsequence of consecutive integers in $f(w)$, that is,
they form a jog.
\end{proof}

In fact, Corollary~\ref{easyfact} extends (by the same proof) to:

\begin{corollary}
\label{easyfact_inv}  If $p$ is an involution, and the entries $a,a+1$ are in ascending order in $p$, then $a,a+1$ will form an inversion (be in descending order) in $f(p)$.
\end{corollary}

A {\em $k$-increasing subsequence} in a permutation is the union of $k$ increasing subsequences. For instance,
if $p=741852963$, then 748596 is a 2-increasing subsequence in $p$ as it is the union of the increasing subsequences 
789 and 456. Similarly, a {\em $k$-decreasing subsequence}
 is a union of $k$ decreasing subsequences. 

The following theorem connects $k$-increasing and $k$-decreasing subsequences of permutations to their 
images by the RS bijection. 

\begin{theorem}[Greene-Fomin-Kleitman, GFK] \label{gfk} \cite{fomin} \cite{greene-kleitman}
Let $p$ be a permutation, and let $a_i$ denote the length of the $i$th row of $P(p)$.
Then for all $k$, the sum $a_1+a_2+\cdots +a_k$ is equal to the length of the longest $k$-increasing subsequence of 
$p$. 

Equivalently, let $b_i$ denote the length of the $i$th column of the $P$-tableau of $p$. Then for all $k$, the sum $b_1+b_2+\cdots
+b_k$ is equal to the length of the longest $k$-decreasing subsequence of $p$. 
\end{theorem}

\begin{corollary} \label{layercor} Let $w$ be a layered permutation, and
let $b_i$ denote the length of the $i$th column of the tableau of $w$. Then $b_i$ is equal to 
the length of the $i$th longest layer of $w$. 
\end{corollary}

\begin{proof} 
A decreasing subsequence cannot contain entries from more than one layer of $w$, so a  $k$-decreasing subsequence
of $w$ must consist of subsequences of $k$ distinct layers of $w$. So, in order to maximize the length of a $k$-decreasing
subsequence, one must choose complete layers, and the $k$ longest ones. 
\end{proof}

Let us keep the notation of Theorem \ref{gfk}. As in any permutation $p$, any $k$ jogs form a $k$-increasing subsequence, it follows that for any permutation $p$, the sum 
$a_1+a_2+\cdots +a_k$ must be at least as large as the combined length of the $k$ longest jogs of $p$. 
This motivates the following definition. 

\begin{definition} We say that a permutation $p$ is {\em GFK-tight} if, for all $k$, the combined length of the $k$
longest jogs of $p$ is equal to the length of the longest $k$-increasing subsequence of $p$. 
\end{definition}

\begin{example}  Consider $p=673481259$.  The jogs of $p$ in descending order of length are $6789, 345, 12$.  These jogs also can be used to create corresponding $k$-increasing subsequences of longest length, namely $6789, 6734859, 673481259$.
\end{example}

The following theorem will show that the images of layered permutations under our involution $f$ are precisely the 
GFK-tight involutions. 

\begin{theorem} \label{characterization} Let $p=p_1p_2\cdots p_n$ be an involution.  Then $p$ satisfies the transposed layer condition
if and only if $p$ is GFK-tight. 
\end{theorem}

\begin{proof}  Let $a_i$ be the length of the $i$th row of the tableau of $p$.

Let us first assume that $p$ satisfies the transposed layer condition.
Then $P(p)^T$ is layered, so it is the $P$-tableau of a layered permutation $w$, whose $i$th column
is of length $a_i$. By Corollary \ref{layercor}, the $k$ longest layers of $w$ have combined length $\sum_{i=1}^ka_i$, 
so by Corollary \ref{easyfact}, the combined length of the $k$ longest jogs of $f(w)=p$ is also  $\sum_{i=1}^ka_i$, 
so by Theorem \ref{gfk}, $p$ is GFK-tight. 

Now let us assume that $p$ is GFK-tight. Note that this implies that the first row of $P(p)$ is as long as the longest
jog of $p$, the second row of $P(p)$ is as long as the second longest jog of $p$, and so on. 

We will show that $P(p)$ satisfies the transposed layer condition by showing that for each jog $A$ of $p$, the
entries that belong to $A$ belong to consecutive columns of $P(p)$, starting with the leftmost column. Let $A_j$ be
the $j$th longest jog of $p$, and let $A_j$ be of length $\ell_j$. We prove our statement by induction on $j$. 

First, let $j=1$. If our claim did not hold for $A_1$, that would mean that the entries of $A_1$ would "skip" a column, 
possibly the first one. As no column can contain more than one entry of $A_1$ (by Remark \ref{inverse}), this would imply that at least one entry of $A_1$ is strictly on the right of the $\ell_1$st column of $P(p)$, which
in turn would imply that the first row of $P(p)$ is longer than $\ell_1$. That would contradict the assumption that $p$ is
GFK-tight since it would imply that the longest increasing subsequence of $p$ is longer than its longest jog. 

Now let us assume that our claim holds for all indices less than $j$, and prove it for $j$. Our conditions then imply that for
any $i<j$, each of the first $\ell_i$ columns of $P(p)$ contains exactly one entry of $A_i$. So 
the entries within each column of $P(p)$ can be rearranged so that the first row will consist of the entries of $A_1$, the 
second row will consist of the entries of $A_2$, and so on, and the $(j-1)$st row will consist of the entries of $A_{j-1}$. 
These rearrangements of the entries do not change 
the column in which any one
entry is located. So, after these rearrangements, the entries that belong to $A_j$ are all below the first $j-1$ rows, and, 
if they skip a column, at least one of them is strictly on the right of the $\ell_j$th column. However, that implies that
the $j$th row of $P(p)$ is longer than $\ell_j$, contradicting the assumption that $p$ is GFK-tight.  
\end{proof}

So if we know that $p\in f(L_n)$, that is, that $p$ is a GFK-tight involution,  then we can obtain $f(p)$ as the unique layered permutation whose layers are identical to the reverses of the jogs of $p$. In order to decide whether $p$ is GFK-tight
or not, it suffices to construct its $P$-tableau and see if it satisfies the transposed layer condition. 

For the sake of completeness, we say that a permutation $p$ is {\em dually GFK-tight} if, for all $k$, the length of the
longest $k$-decreasing subsequence of $p$ is equal to the combined length of the $k$ longest reverse jogs in $p$, 
where a reverse jog is a nonextendible decreasing subsequence of consecutive integers. This leads to the analogous
version of Theorem \ref{characterization}. 

\begin{theorem} A permutation $p$ is layered if and only if it is a dually GFK-tight involution.
\end{theorem}

This is the new characterization of layered permutations that we promised at the beginning of this section.

The results in this section raise an intriguing question. Let us consider {\em permutations} $p$ for which both
$P(p)$ and $Q(p)$ satisfy the layer condition or the transposed layer condition, but they are not necessarily identical.
 How can we describe these permutations? 
We will return to these questions in Section \ref{generalized}.

\section{Involutions that avoid 321 or 123} \label{catalan}

If a permutation does not contain a decreasing (resp. increasing) subsequence of length three, then we say that it
{\em avoids} the pattern 321 (resp. 123). 

It follows from Theorem \ref{gfk} that if an involution $p$ avoids the pattern 321, then its tableau consists of at most
two rows. Therefore, the image $f(p)$ consists of at most two columns. 

It turns out that given $P(p)$, we can recover the involution $p$ without running the inverse of the RS bijection. 
In an involution, each element is either a fixed point, or part of a 2-cycle. If $(i \ j)$ is a 2-cycle, and $i<j$, then we call 
$i$ a {\em small entry} and $j$ a {\em large entry}. So each entry of an involution is either a fixed point, or a small entry, or a large entry. 

\begin{proposition}~\label{2_row}  Let $p$ be a 321-avoiding involution. Then the first row of $P(p)$  consists of all the small entries and all the fixed points, and the second row consists of all the large entries. 
\end{proposition}

\begin{proof}
A small entry can never be displaced during the formation of $P(p)$. Indeed, if $(ab)$ is a $2$-cycle of $p$, $a$ is the small entry in it, and later on, $a$ is displaced by $x$, then $bax$ is a $321$-pattern in $p$. 

Also, a fixed point $y$ can never be bumped. If such a $y$ is bumped by $j$, then $j<y$, and $j$ follows $y$. So $j=p_i$, where $i>y$.  However, this implies that  $(i\ j)$ is a $2$-cycle, so $p_j=i$, and $iyj$ is a $321$-pattern.

On the other hand, one can see that all large entries will be bumped. By way of contradiction, suppose at least one large entry is not bumped to the second row.  Let $k$ be the number of fixed points, let $m$ be the number of $2$-cycles. Then the first row would be
at least $k+m+1$ entries long, which would imply, by Theorem \ref{gfk}, the existence of an increasing subsequence of length $k+m+1$. That is impossible, since that would imply that there is a $2$-cycle whose entries both belong to that increasing subsequence. 
\end{proof}

We will now describe a way to directly recover the 321-avoiding involution $p$ of length $n$ from its tableau $P(p)$. 

If $n$ is in the first row, then $n$ is a fixed point.  Remove $n$ from $P(p)$, and continue with $n-1$.

If instead, $n$ is in the second row, then  $n$ must be a large entry in $p$. 

We claim $n$ must be in a $2$-cycle with the largest entry, say $k$ in the first row.  If this were not the case, then $n$ is in a $2$-cycle with $j$ where $j<k$.  We then have two options.  One, $k$ is a fixed point which means $n k j$ is a $321$ pattern in $p$.  Or two, $k$ is the small entry of a $2$-cycle with $m$ where $k < m < n$ and then $nmkj$ is a $4321$ pattern in $p$.  Hence, we know $n$ must be in a $2$-cycle with $k$ and we remove $n$ and $k$ from $P(p)$. 

We can then continue this process with the next largest entry remaining in $P(p)$.

\begin{theorem}~\label{reverse_321}  To obtain a $321$-avoiding permutation $p$ from $P(p)$, while $P(p) \neq \emptyset$, let $n$ be the maximum of all entries in $P(p)$.  If $n$ is in the top row, then $n$ is a fixed point of $p$ and we remove $n$ from $P(p)$.  Else, $n$ is in a $2$-cycle with $k$ where is the maximum of all entries in row one of $P(p)$.  Remove $n$ and $k$ from $P(p)$ and continue this process with the next largest entry remaining in $P(p)$.
\end{theorem}

\begin{example}  Let

$P(p) =\ytableausetup{centertableaux}
\begin{ytableau}
1 & 2 & 4 & 6 & 7\\
3 & 5 
\end{ytableau}.$

Then we first note that $7$ is a fixed point, then we note that 6 is a fixed point, after which we recover the $2$-cycles
$(4 \ 5)$ and $(2 \ 3)$, and the fixed point $1$, to obtain the involution $p=(1)(32)(54)(6)(7)=1325467$.
\end{example}

Therefore, if instead we have a $123$-avoiding involution $q$, and we know $Q(q)$, it is simple to compute $f(q)$ directly. In fact, 
knowing $Q(q)$ is equivalent to knowing its first column. For the $123$-avoiding involution $q=q_1q_2\cdots q_n$,
let us call the index $i$ a {\em record-breaker} if the longest decreasing subsequence of the initial segment 
$q_1q_2\cdots q_i$ is longer than that of $q_1\cdots q_{i-1}$. It then follows from Theorem \ref{gfk} that the 
entries in the first column of $Q(q)$ are precisely the record breakers of $q$. 

So, if $q$ is a $123$-avoiding involution, then we can first find its record-breakers, turn them into the set of 
small entries and fixed points of $f(q)$, turn the remaining entries of $q$ into the set of large entries of $f(q)$, and
finally match the small and large entries of $f(q)$ as explained in Theorem \ref{reverse_321}. 

\begin{example}
Let $q=6574213$, then the record-breakers of $q$ are $1, 2, 4, 5,$ and $6$. So the large entries of $f(q)$ are $3$ and $7$, 
while the other entries of $f(q)$ are fixed points or small entries. Therefore, Theorem \ref{reverse_321} explains that
$f(q)=  (1)(32)(4)(5)(76)=1324576$.
\end{example}

\section{A generalization of layered permutations} \label{generalized}

In this section, we are turning our attention to certain {\em permutations} instead of involutions. Our main 
result is the following. 

\begin{theorem} \label{general}  For any permutation $p$, the following two statements are equivalent. 
\begin{enumerate}
\item[(A)] Both $P(p)$ and $Q(p)$ satisfy the transposed layer conditions.
\item[(B)] Both $p$ and $p^{-1}$ are GFK-tight.
\end{enumerate}
\end{theorem}

Note that in general, it is not true that $p$ is GFK-tight if and only if $p^{-1}$ is GFK-tight. While the shapes of the 
tableaux of $p$ and $p^{-1}$ are always the same, the lengths of their jogs may not be. For instance, $p=1423$ is
GFK-tight, but $p^{-1}=1342$ is not. Also note that if both $p$ and $p^{-1}$ are GFK-tight, then these two permutations
must have the same number of jogs of each length $\ell$, namely, the number of rows of length $\ell$ of their tableaux. 

\begin{proof}
\begin{enumerate}
\item[(A)] Let us assume that $P(p)$ and $Q(p)$ both satisfy the transposed layer conditions. Let $g_1=1,g_2,\cdots ,g_t$
be the entries in the first column of $Q(p)$. Then by Proposition \ref{descent}, the permutation $p^{-1}$ has $t$
jogs, one starting in each $g_i$. If an entry $x$ is the $j$th entry of its jog of $p^{-1}$, then the transposed layer condition implies
that $x$ appears in the $j$th column of $Q(p)$. Now let us assume that $p^{-1}$ is {\em not} GFK-tight. Then there is a
$k$ so that the longest $k$-increasing subsequence is longer then the combined length of the $k$ longest jogs of
$p^{-1}$. Choose the smallest such $k$, then the $k$th row of $Q(p)$ is of length $a_k$, whereas the $k$th longest jog
of $p^{-1}$ is of length $\ell_k$, with $a_k>\ell_k$. As the last entry of this row must be the $a_k$th entry of its
jog in $p^{-1}$, it follows that there are at least $k$ jogs in $p^{-1}$ that are of length $a_k$ or more, contradicting
the inequality $a_k>\ell_k$. This proves that $p^{-1}$ is GFK-tight.

Replacing $Q(p)$ by $P(p)$ and $p^{-1}$ by $p$, we get an analogous proof of the fact that $p$ is GFK-tight. 

\item[(B)] Let us assume that $p$ and $p^{-1}$ are both GFK-tight. It then follows that for all $i$, the $i$th row of $P(p)$ and
$Q(p)$ as of length $\ell_i$, the length of the $i$th longest jog of $p$. For any jog $A_j$ of $p$, no two entries of 
$A_j$ can be in the same column of $Q(p)$, since that would imply, by Remark \ref{inverse}, that there is an entry $i+1$ 
of that jog that precedes the entry $i$ in $p^{-1}$. From this, we can prove that $Q(p)$ satisfies the transposed layer
condition as we did in the proof of Theorem \ref{characterization}. Similarly, for any jog $B_j$ of $p^{-1}$, 
no two entries of 
$B_j$ can be in the same column of $Q(p^{-1})=P(p)$, since that would imply, by Remark \ref{inverse}, that there is an entry $i+1$ 
of $B_j$ that precedes the entry $i$ in $p$. Then we can prove that $P(p)$  satisfies the transposed layer
condition in the same way.
\end{enumerate}
\end{proof}

\begin{corollary} \label{genlay} 
For any permutation $p$, the following two statements are equivalent. 
\begin{enumerate}
\item[(A)] Both $P(p)$ and $Q(p)$ are layered.
\item[(B)] Both $p$ and $p^{-1}$ are dually GFK-tight.
\end{enumerate}
\end{corollary}

It is natural to ask how many permutations are of the kind that is described by Theorem \ref{general}, or, equivalently, 
by Corollary \ref{genlay}. It is easier (in terms of terminology) to discuss the answer for the latter. The first question we must answer is how many 
layered SYT of a given shape are there? As layered SYT with column lengths $b_1,b_2, \cdots, b_k$ are in bijection
with layered permutations of layer lengths $b_1,b_2, \cdots, b_k$, it follows that the number of such SYT is equal to 
the number of distinct multiset-permutations of the multiset $\{b_1,b_2,\cdots ,b_k\}$.

In order to announce our formula, we need one definition. For a partition $h$ of the integer $n$, let 
$\hbox{comp}(h)$ denote the total number of compositions of $n$ that are obtained by rearranging the parts of $h$. 
For instance, if $h=3+2+1$, then $\hbox{comp}(h)=6$, while if $h=2+2+1+1+1$, then  $\hbox{comp}(h)=10$.

So, there are $\hbox{comp}(h)^2$ pairs of layered SYT of shape $h$, and we proved the following theorem. 
\begin{theorem}
The number of permutations $p$ of length $n$ so that both $p$ and $p^{-1}$ are dually GFK-tight (equivalently, GFK-tight),
is 
\[A_n=\sum_h \hbox{comp}(h)^2,\]
where $h$ ranges all partitions of the integer $n$. 
\end{theorem} 

Note that the sequence of the numbers of $A_n$ is in OEIS, as sequence A263897. It is obvious that it also counts
{\em anagram compositions} of $2n$, that is, compositions of $2n$ into $2k$ parts, so that the multiset of the first
$k$ parts is identical to the multiset of the last $k$ parts. It would be interesting to find out how large the numbers 
$A_n$ are. It is clear that their exponential order is 4, since, on the one hand, 
\[A_n=\sum_h \hbox{comp}(h)^2 \leq \left ( \sum_h \hbox{comp}(h) \right )^2 =4^{n-1},\]
and on the other hand, by the Cauchy-Schwarz inequality, 
\[A_n=\sum_h \hbox{comp}(h)^2 \geq \frac{1}{p(n)} \left (\sum_h \hbox{comp}(h) \right )^2 =\frac{4^{n-1}}{p(n)},\]
where $p(n)$ denotes the number of partitions of $n$. As it is well known that $p(n)$ is of exponential order 1, 
our claim is proved. 

As dually GFK-tight involutions have a very simple characterization (they are the layered permutations), it is natural to 
ask if GFK-tight involutions, as well as permutations described in Theorem \ref{general} and Corollary \ref{genlay} can
be described in a simpler way.

\bibliographystyle{abbrv}
\bibliography{invol-bib}

\end{document}